\documentclass{amsart}

\usepackage{amsfonts,amsmath, amsthm, amssymb, latexsym, epsfig}
\usepackage[english]{babel}
\usepackage[alphabetic]{amsrefs}
\usepackage[utf8]{inputenc}
\usepackage[all]{xy}
\usepackage{xspace}
\usepackage{amsmath}
\usepackage{amscd}
\usepackage{enumitem}
\usepackage{comment}
\usepackage{setspace}
\usepackage{stmaryrd}
\usepackage{xcolor}
\usepackage{tikz-cd}
\usepackage{hyperref}
\usepackage[mathscr]{eucal}
\usepackage{mathtools}

\numberwithin{equation}{section}

\usepackage[hmargin=3cm,vmargin=3cm]{geometry}

\newcommand{\C}{{\mathbb{C}}}

\newcommand{\ctone}{{C_{\mathtt{t1}}}}

\newcommand{\M}{{\mathrm{M}}}

\newcommand{\mB}{{\mathcal{B}}}
\newcommand{\B}{{\mathrm{B}}}
\newcommand{\nf}{{2,d}}

\newcommand{\R}{{\mathbb{R}}}

\renewcommand{\l}{\left}
\renewcommand{\r}{\right}

\newcommand{\bee}{\begin{equation}}
\newcommand{\ene}{\end{equation}}
\newcommand{\bone}{\mathbf{1}}

\newcommand{\br}{\operatorname{br}}
\newcommand{\diag}{\operatorname{diag}}
\newcommand{\tr}{\operatorname{tr}}

	\newtheorem{thm}{Theorem}[section]
	\newtheorem{cor}[thm]{Corollary}
	\newtheorem{lem}[thm]{Lemma}

	\theoremstyle{definition}

	\theoremstyle{remark}
	\newtheorem{rmk}[thm]{Remark}

\title{On almost commuting matrices with respect to the normalized Hilbert--Schmidt norm}
\author{Mohit Bansil}
\address{
}
\thanks{The present work was completed while M. B. was an undergraduate student at the Department of Mathematics, Michigan State University, East Lansing, and a graduate student at University of California, Los Angeles}
\email{mbansil@g.ucla.edu}
\author{Ilya Kachkovskiy}
\address{Department of Mathematics\\ Michigan State University\\
Wells Hall, 619 Red Cedar Rd\\ East Lansing, MI\\ 48824\\ USA}
\email{ikachkov@msu.edu}
\thanks{I. K. was partially supported by the NSF grants DMS--1846114, DMS--2052519, and the 2022 Sloan Research Fellowship.}
\date{}

\begin{document}

\begin{abstract}
In this paper, we consider the Rosenthal -- Halmos problem of almost commuting matrices with respect to the normalized Hilbert -- Schmidt norm $\|\cdot\|_{2,d}=d^{-1/2}\|\cdot\|_2$. We show that if $X$ and $Y$ are self-adjoint matrices with $\|X\|\le 1$, $\|Y\|\le 1$, then there exist commuting self-adjoint matrices $X',Y'$ such that $\|X-X'\|_\nf+\|Y-Y'\|_\nf\le 5\|[X,Y]\|_{2,d}^{1/3}$, and $[X,X']=0$. Within these constraints, the exponent $1/3$ cannot be improved.
\end{abstract}

\maketitle

\section{Introduction and main results}
The topic of almost commuting matrices dates back to Rosenthal \cite{Rosenthal} and Halmos \cite{Halmos}, with the original question that can be traced back to von Neumann \cite{Neumann,Neumann_translation}.
Suppose, $X$ and $Y$ are two self-adjoint matrices such their commutator $[X,Y]$ is small in some way. Can one find commuting matrices $X',Y'$ that are close to $X$ and $Y$, respectively, in the same way? Many versions and generalizations of this question exist, depending on the notion of smallness and specific types of matrices, with two main paths to obtaining positive results. Abstract approaches usually involve proofs by contradiction. That is, non-existence of a dimension-independent estimate leads to an element with impossible properties in a certain $C^*$-algebra associated to sequences of matrices. Quantitative approaches, on the other hand, usually deal with the original matrices directly, in a more model-specific way, but are often still motivated by $C^*$-algebraic intuition.

In the case when smallness/closeness is understood in terms of the operator norm, a major breakthrough was obtained by Lin in \cite{Lin_main}, resulting in an abstract approach implying the the existence of a dimension-independent distance estimate to commuting pairs in terms of $\|[X,Y]\|$. The proof was later simplified in \cite{FR}. A quantitative result demonstrating existence of $X',Y'$ such that
\bee
\label{eq_exponent}
\|X-X'\|+\|Y-Y'\|\le C\|[X,Y]\|^{1/2}
\ene
was obtained much later in \cite{KS} (see also \cite{H_orig,Herrera}), where the exponent $1/2$ cannot be improved further.

The original question of Rosenthal was posed in terms of the Hilbert -- Schmidt norm. For the purpose of dimension-independent estimates, one naturally arrives to {\it normalized Hilbert--Schmidt norm}
\bee
\label{eq_nf_def}
\|X\|_{\nf}^2:=\frac{1}{d}\sum_{i,j=1}^d |A_{ij}|^2,\quad X\in \M_d.
\ene
The first result that matches Lin's theorem in this setting was obtained in \cite{H0} (see also \cite{H1,H2,H3}), in a bigger generality that allows to deal with an arbitrary finite number of almost commuting matrices. An independent proof for two matrices was obtained in \cite{FS}. In both cases, the proofs are of an abstract existence-only kind. The first quantitative result involving the norm \eqref{eq_nf_def} was obtained in \cite{Gl}, providing an analogue of \eqref{eq_exponent} with the exponent $1/6$ in the right hand side. The exponent was later improved to $1/4$ in \cite{FS} (see also  \cite{Said,Geisler}). In these results, the key feature is that $X'$ commutes with $X$, which allows to iterate the construction in order to extend the estimates to a larger number of almost commuting matrices with meaningful distance estimates.

In the present paper, we study the Rosenthal -- Halmos problem with respect to the norm \eqref{eq_nf_def}, under the additional constraint that $X'$ must commute with $X$, motivated by \cite{Gl,FK}. Our main result is Theorem \ref{th_main} below, which obtains a distance estimate with an exponent $1/3$. As Theorem \ref{th_lower_bound} shows, this exponent is sharp. Unlike \cite{Gl,FK}, the conclusion of Theorem \ref{th_main} cannot be immediately iterated in order to cover more than two almost commuting matrices. However, an appropriate modification is obtained in our Theorem \ref{th_main_multi1}.
\subsection{Results on two matrices}
Let $\M_d=\M_d(\C)$ be the space of $(d\times d)$-matrices with complex entries. Define the normalized Hilbert--Schmidt norm as above in \eqref{eq_nf_def}. Clearly, we have
\bee
\label{eq_norms}
d^{-1/2}\|X\|_2=\|X\|_\nf\le \|X\|,
\ene
where $\|X\|_2$ is the usual Hilbert--Schmidt (Frobenius) norm, and $\|X\|$ is the operator norm. Let also
$$
\B_d:=\{X\in \M_d(\C)\colon X=X^*,\,\|X\|\le 1\}
$$
be the unit ball in the space of self-adjoint matrices from $\M_d$, with respect to the operator norm. The following is the main result. Due the quadratic nature of the the Hilbert--Schmidt norms, it is convenient to square the corresponding distances.
\begin{thm}
\label{th_main}
For any $X,Y\in \B_d$ there exist $X',Y'\in \B_d$ such that $[X',Y']=[X',X]=0$ and
\bee
\label{eq_th_main}
\|X-X'\|_\nf^2+\|Y-Y'\|_\nf^2\le \ctone\|[X,Y]\|_\nf^{2/3}.
\ene
\end{thm}
In order to better follow the structure of the proof, we believe it is convenient to carry $\ctone$ through various intermediate steps. However, an inspection of the proof shows that one can take $\ctone=24$ which, after rounding, leads to the constant $5$ in the abstract. We note that no specific optimization efforts were made.

We also show that the estimate \eqref{eq_th_main} is order sharp, as follows.
\begin{thm}
\label{th_lower_bound}
Let
\bee
\label{eq_th_lower_bound}
X_0:= \frac{1}{d} 
\begin{pmatrix}
1 &  &  &  \\
& 2 &  &  \\
&  & \ddots &  \\
&  &  & d
\end{pmatrix}
, \quad
Y_0 := 
\begin{pmatrix}
0 & 1 &  &  \\
1& 0 & \ddots &  \\
& \ddots & \ddots & 1 \\
&  & 1 & 0
\end{pmatrix}\in \B_d.
\ene
For every $X',Y'\in \B_d$ satisfying $[X',Y']=[X',X_0]=0$, we have
$$
\|X_0-X'\|_\nf^2+\|Y_0-Y'\|_\nf^2\ge \frac{1}{96} \|[X,Y]\|_\nf^{2/3}.
$$
\end{thm}
\subsection{Three or more matrices}The complexity of operator norm version of the Rosenthal -- Halmos problem increases drastically, once one departs from the case of two self-adjoint matrices. Indeed, topological obstructions appear in the case of two unitary matrices \cite{Voiculescu,Eilers1,Gong_Lin,DHK} as well as in the case of three self-adjoint matrices \cite{Choi,HL}, and the situation is only expected to become harder in more involved cases. On the other hand, as first observed in \cite{H0}, many of these obstructions do not appear in the case of the normalized Hilbert -- Schmidt norm (however, we refer the reader to \cite{Atkinson,Ioana} for some more delicate cases). As was noted in \cite[Theorem 3]{Gl} (see also \cite[Theorem 3]{FK}), the constructions of \cite{Gl} and \cite{FK} can be iterated in order to obtain quantitative bounds for arbitrary number of mutually almost commuting matrices.

The above iterations rely on the fact that the operator $X'$ in \cite{Gl,FK} does not depend on $Y$, which does not apply directly to our case of Theorem \ref{th_main}. However, one can observe by inspecting the proof of Theorem \ref{th_main} that the discrete optimization part of the proof can be modified in order to treat a $k$-tuple of matrices $Y$ instead of a single one, as follows.
\begin{thm}
\label{th_main_multi1}
For any $X,Y_1,\ldots,Y_N\in \B_d$, there exist $X',Y_1',\ldots,Y_N'\in \B_d$ such that
$$
[X',Y_1']=\ldots=[X',Y_N']=[X',X]=0,
$$
and
$$
\|X-X'\|_\nf^2+\|Y_1-Y_1'\|_\nf^2+\ldots+\|Y_N-Y_N'\|_\nf^2\le 24\sqrt{N}\l(\|[X,Y_1]\|^2_\nf+\ldots+\|[X,Y_N]\|^2_\nf\r)^{1/3}.
$$
\end{thm}
As mentioned above, the proof of Theorem \ref{th_main_multi1} follows the proof of Theorem \ref{th_main}. In order to avoid both unnecessary repetitions and over-generalizing the main result, we explain the relevant modifications after each step of the proof of Theorem \ref{th_main}, in Remarks \ref{rem_optimization_multiple}, \ref{rem_dimension_multiple}, and \ref{rem_main_multiple}.

Similarly to \cite{Gl,FK}, the reader can now iterate Theorem \ref{th_main_multi1} for the matrices $Y_1',\ldots,Y_N'$ restricted into the eigenspaces of $X$, without violating commutation with $X$. 
\begin{cor}
\label{cor_main_multi2} For any $X_1,\ldots,X_N\in \B_d$, there exist commuting matrices $X_1',\ldots,X_N'\in \B_d$ such that
\bee
\label{eq_inequality_multiple}
\|X-X_1'\|_\nf^2+\ldots+\|X_N-X_N'\|_\nf^2\le C(N)\max_{i,j}\l(\|[X_i,X_j]\|_\nf\r)^{\frac{2}{3^{N-1}}},
\ene
where $C(N)$ only depends on $N$.
\end{cor}
While \eqref{eq_inequality_multiple} improves on the previously known results, we do not expect it to be optimal.
\subsection{Structure of the paper} In Section 2 we show that, if one additionally prescribes the multiplicities of the eigenvalues of $X'$, then the of $X'$ and $Y'$ can be obtained by an explicitly solvable optimization problem, thus reducing the original question to optimizing the choice of said multiplicities, which is a discrete optimization problem.

In Section 3 we discuss two technical lemmas, which collectively conclude that, without loss of generality one can assume that $X$ has no large spectral gaps.

In Section 4 we prove the main result, Theorem \ref{th_main}. The main idea is based on carefully estimating a particularly chosen weighted average of the solution candidates to the discrete optimization problem, thus providing an upper bound on the actual minimizer.

In Section 5, we use the ideas obtained from Section 2 to prove Theorem \ref{th_lower_bound}, which provides a lower bound on the exponent in Theorem \ref{th_main}.

As mentioned in the previous subsections, the proof of Theorem \ref{th_main_multi1} is contained in Remarks \ref{rem_optimization_multiple}, \ref{rem_dimension_multiple}, and \ref{rem_main_multiple}
With all the preparations, the proof of Corollary \ref{cor_main_multi2} is a straightforward consequence of the iteration argument described in the same subsection. We refer the reader to \cite{FK} for additional details.

\section{An optimization solution}
Since $X'$ commutes with $X$, they can be jointly diagonalized by a unitary conjugation. As the conjugation preserves all norms under consideration, {\it one can assume without loss of generality that both $X$ and $X'$ are diagonal}. Additionally, by approximation, one can also assume that all eigenvalues of $X$ are distinct. Therefore, let us denote
$$
\B^{\diag}_d:=\{X\in \M_d(\C)\colon X=X^*,\,\|X\|\le 1,\,\,X\,\text{is diagonal with simple spectrum}\}.
$$
Let also
$$
X'=\diag\{t_1,\ldots,t_d\},\quad -1\le t_1\le\ldots\le t_d\le 1.
$$
Unlike $X$, the eigenvalues of $X'$ will in general not be simple. Indeed, otherwise $Y'$ would also have to be diagonal, which is too wasteful for the purpose of minimizing $\|Y-Y'\|_\nf$. It is convenient to split the set of all $X'$ under consideration into classes based on the eigenvalue multiplicities. More precisely, if
$$
t_1=\ldots=t_{b_1}<t_{b_1+1}=\ldots=t_{b_2}<\ldots<t_{b_{\ell-1}+1}=\ldots=t_{b_\ell},\quad b_\ell=d,
$$
then we will call $b_1,\ldots,b_\ell$ the {\it breaking points} of $X'$, denoted by $\br(X')=(b_1,\ldots,b_\ell)$. Clearly, $t_{b_1},\ldots,t_{b_\ell}$ are distinct eigenvalues of $X'$, with multiplicities $b_1-b_0,\ldots,b_\ell-b_{\ell-1}$, respectively, where it is assumed for convenience that $b_0=0$. Let also $P_j$ be the spectral projection of $X'$ associated to the $j$-th eigenvalue with the eigenspace $\ker(X'-t_{b_j}\bone)$.

Denote by $F$ the left hand side of \eqref{eq_th_main}, considered as a functional of $X'$ and $Y'$:
\bee
\label{eq_F_def}
F(X',Y'):=\|X-X'\|_2^2+\|Y-Y'\|_2^2.
\ene
The following lemma states that, assuming that the breaking points of $X'$ are given, the matrices $X'$ and $Y'$ minimizing \eqref{eq_F_def} can be explicitly found.
\begin{lem}
\label{lemma_reduction}
Fix $X\in \B_d^{\diag}$, $Y\in \B_d$, and a sequence of breaking points $\mB=(b_1,\ldots,b_\ell)$. Then, the functional $\eqref{eq_F_def}$ has a unique minimizer
$$
(X'_\mB,Y'_\mB)=\arg\min \l\{F(X',Y')\colon X'\in \B_d^{\diag},\,\br(X')=\mB;\,Y'\in \B_d\,;\,[X',Y']=0\r\}
$$
given by
\bee
\label{eq_minimizer_x}
X_\mB'=\sum_{j=1}^{\ell}\frac{1}{b_{j}-b_{j-1}}\tr(P_j X P_j)P_j;
\ene
\bee
\label{eq_minimizer_y}
Y'_\mB=\sum_{j=1}^\ell P_j Y P_j.
\ene
\end{lem}
\begin{proof}
Note that $P_j$ only depends on $\mB$, and is therefore fixed within the considered range of $X',Y'$. Clearly, $Y'$ commutes with $X'$ if and only if it commutes with every $P_j$. The latter is equivalent to $P_i Y' P_j=0$ for all $i\neq j$. The vanishing of the associated matrix elements of $Y'$ leads to
$$
\|Y-Y'\|_2^2\ge \sum_{i,j\colon i\neq j}\|P_i Y P_j\|^2_2.
$$
Clearly, the equality is attained on $Y'=Y'_\mB$ defined in \eqref{eq_minimizer_y}.

Recall that, for $y_1,\ldots,y_k\in \R$, we have
\bee
\label{eq_quadratic_min}
\arg\min\l\{(y_1-t)^2+\ldots+(y_k-t)^2\colon t\in \R\r\}=\frac{y_1+\ldots+y_k}{k}.
\ene
Since $\mB$ is fixed, the matrix $X'$ is completely determined by $t_{b_1},\ldots,t_{b_\ell}$, and one can rewrite the functional \eqref{eq_F_def} as
$$
F(X',Y')=\sum_{i=1}^{\ell}\sum_{j=b_{i-1}+1}^{b_i}(x_j-t_{b_i})^2+\|Y-Y'_\mB\|^2_2,
$$
and one can now obtain \eqref{eq_minimizer_x} as a consequence of \eqref{eq_quadratic_min} applied with $t=t_{b_i}$ separately for each $i$. Note that the assumption of all $x_j$ being distinct implies that the corresponding minimizers $\lambda_j$ will also be distinct, thus satisfying the assumption $\br(X'_\mB)=\mB$.
\end{proof}
\begin{rmk}
\label{rem_optimization_multiple}
The conclusion of Lemma \ref{lemma_reduction} carries to the setting of Theorem \ref{th_main_multi1} verbatim, applied to $Y=Y_j$ for each $j$ independently.
\end{rmk}
\section{Dimension reductions}
In this section, we will describe two cases in which the problem can be reduced to that in a smaller dimension. The first lemma is, essentially a restatement of the fact that $\|\cdot\|_\nf$ ``scales correctly'' with the dimension for the purpose of the question of almost commuting matrices.
\begin{lem}
\label{lemma_direct_sum}
Suppose that $X=X_-\oplus X_+$, $Y=Y_-\oplus Y_+$ are decomposable with respect to some orthogonal decomposition of the original space, and $(X_-',Y_-')$, $(X_+',Y_+')$ satisfy the conclusions of Theorem $\ref{th_main}$ on the corresponding subspaces. Then $X'=X_-'\oplus X_+'$ and $Y'=Y_-'\oplus Y_-'$ satisfy the same conclusion for the original matrices $X,Y$.
\end{lem}
\begin{proof} Let $d_1+d_2=d$ be the dimensions of the corresponding spaces. From the conclusion of Theorem \ref{th_main} applied to the orthogonal summands, the relation \eqref{eq_norms} between the norms, and H\"older's inequality, we have
\begin{multline}
\label{eq_inqualities_directsum}
\|X-X'\|^2_2+\|Y-Y'\|^2_2=\|X_--X_-'\|^2_2+\|Y_--Y_-'\|^2_2+\|X_+-X_+'\|^2_2+\|Y_+-Y_+'\|^2_2\le \\ \le C_{\mathrm{t1}}\l(d_1^{2/3}\|[X_-',Y_-']\|_2^{2/3}+d_2^{2/3}\|[X_-',Y_-']\|^{2/3}_2\r)\le \\ \le C_{\mathrm{t1}}\l(d_1^{\frac23\cdot\frac32}+d_2^{\frac23\cdot\frac32}\r)^{2/3}\l(\|[X_-',Y_-']\|_2^2+\|[X_+',Y_+']\|_2^2\r)^{1/3}=C_{\mathrm{t1}}d^{2/3}\|[X,Y]\|^{2/3}_2,
\end{multline}
and the proof can be completed after multiplying both sides by $d^{-1}$ and using \eqref{eq_norms}.
\end{proof}
The second property is somewhat more subtle and indicates that certain matrix elements of $Y$ that are away from the diagonal can be assumed to vanish without loss of generality.
\begin{lem}
\label{lemma_delete}
Under the notation from Lemma $\ref{lemma_reduction}$, suppose that for some indices $i,j$ we have
\bee
\label{eq_lemma_delete}
\|[X,Y]\|_2^2\le \frac{C_{\mathrm{t1}}^{3/2}d}{6^{3/2}}|x_i-x_j|^3.
\ene
Denote by $\tilde Y$ the matrix $Y$ with both $Y_{ij}$ and $Y_{ji}$ replaced by $0$. Assume that $X',Y'$ satisfy the conclusion of Theorem $\ref{th_main}$ for the matrices $X,\tilde Y$. Then, they also satisfy the same conclusion for the matrices $X,Y$.
\end{lem}
\begin{proof}
Since $\|[X,\tilde Y]\|_2\le \|[X,Y]\|_2$, we have
\begin{multline*}
2|Y_{ij}|^2(x_i-x_j)^2=\|[X,Y]\|^2_2-\|[X,\tilde Y]\|^2_2=\\=\l(\|[X,Y]\|^{2/3}_2-\|[X,\tilde Y]\|^{2/3}_2\r)\l(\|[X,Y]\|^{4/3}_2+\|[X,Y]\|^{2/3}_2\|[X,\tilde Y]\|_2^{2/3}+\|[X,\tilde Y]\|_2^{4/3}\r)\le\\
\le 6 \l(\|[X,Y]\|^{2/3}_2-\|[X,\tilde Y]\|^{2/3}_2\r)\|[X,Y]\|^{4/3}_2.
\end{multline*}
Together with \eqref{eq_lemma_delete}, we get
\bee
\label{eq_delete_1}
2|Y_{ij}|^2\le \frac{6\l(\|[X,Y]\|^{2/3}_2-\|[X,\tilde Y]\|^{2/3}_2\r)\|[X,Y]\|^{4/3}_2}{(x_i-x_j)^{2}}\le C_{\mathrm{t1}}d^{2/3}\l(\|[X,Y]\|^{2/3}_2-\|[X,\tilde Y]\|^{2/3}_2\r).
\ene
By combining the assumptions of the theorem, the relation between $Y$ and $\tilde Y$, and \eqref{eq_delete_1}, we finally arrive to
\begin{multline*}
\|X-X'\|_2^2+\|Y'-Y\|_2^2=\|X-X'\|_2^2+\|Y'-\tilde Y\|_2^2+\|\tilde Y-Y\|_2^2\le \\ \le C_{\mathrm{t1}} d^{2/3}\l\|[X,\tilde Y]\r\|_2^{2/3}+2\l|Y_{ij}\r|^2\le C_{\mathrm{t1}}d^{2/3}\|[X,Y]\|_2^{2/3},
\end{multline*}
which completes the proof after multiplying both sides by $d^{-1}$.
\end{proof}
\begin{rmk}
\label{rem_iterated_delete}
Since $\|[X,\tilde Y]\|_2\le \|[X,Y]\|_2$ in the above proof, deleting a matrix element $Y_{ij}$ would not reverse the inequality \eqref{eq_lemma_delete} for any other pairs of indices for which it was true originally.
\end{rmk}
As a consequence, we can assume without loss of generality that the spectrum of $X$ does not contain large gaps, as follows.
\begin{cor}
\label{cor_gaps}
Suppose that the statement of Theorem $\ref{th_main}$ has been obtained under the additional assumptions
$$
X=\diag\{x_1,\ldots,x_d\},\quad -1\le x_1<\ldots< x_d\le 1;\quad x_{b+1}-x_b < \sqrt{\frac{6}{C_{\mathrm{t1}}}}\|[X,Y]\|_\nf^{2/3},\quad\text{for}\quad b=1,\ldots,d-1.
$$
Then the statement of Theorem $\ref{th_main}$ also holds in full generality.
\end{cor}
\begin{proof}
The assumption of $X$ being diagonal with distinct eigenvalues has already beed discussed in the beginning of Section 2.

Suppose now that, for some $a$, we have
$$
x_{b+1}-x_b\ge \sqrt{\frac{6}{C_{\mathrm{t1}}}}\|[X,Y]\|^{2/3}_\nf,
$$
In that case, one also has
$$
x_{i}-x_j\ge \sqrt{\frac{6}{C_{\mathrm{t1}}}}\|[X,Y]\|^{2/3}_\nf,\quad \text{for}\quad i\ge b+1>b\ge j.
$$
From Lemma \ref{lemma_delete}, one can assume without loss of generality that $Y_{ij}=0$ for all $i,j$ as above. However, in this case we fall within the assumptions of Lemma \ref{lemma_direct_sum} and can therefore proceed by induction.
\end{proof}
\begin{rmk}
\label{rem_dimension_multiple}
Lemma \ref{lemma_direct_sum} applies to the setting of \ref{th_main_multi1} assuming that the subspace decomposition is the same for each pair $(X,Y_1),\ldots,(X,Y_N)$. Note that the $\pm$ notation is chosen in order to avoid conflict with the notation in Theorem \ref{th_main_multi1}.

Lemma \ref{lemma_delete} can be applied to each matrix $Y_j$ individually, in view of Remark \ref{rem_iterated_delete}. However, the dependence on $Y_j$ would not allow to apply Lemma \ref{lemma_direct_sum} in the end of the proof of Corollary \ref{cor_gaps} verbatim. For the purpose of Theorem \ref{th_main_multi1}, an appropriate modification is:
\bee
\label{eq_replacement}
\text{replace}\quad \|[X,Y]\|_\nf \quad\text{by}\quad\l(\|[X,Y_1]\|^2_\nf+\ldots+\|[X,Y_N]\|^2_\nf\r)^{1/2}
\ene
in the statements of Lemma \ref{lemma_delete} and Corollary \ref{cor_gaps}.
\end{rmk}
\section{Proof of Theorem \ref{th_main}}
We will use $C_{\mathrm{t1}}$ in some of the estimates, and will assign to it a numerical value 12 retrospectively. As earlier in Lemma \ref{lemma_reduction} and Corollary \ref{cor_gaps}, we assume
$$
X=\diag\{x_1,\ldots,x_d\},\quad -1=x_1<\ldots< x_d=1.
$$
Note that the equalities above can always be imposed by shift and rescaling of $X$. Recall that, by Lemma \ref{lemma_reduction}, the best commuting approximants can be explicitly found once one fixes the breaking points:
$$
X'=\diag\{t_1,\ldots,t_d\},\quad -1\le t_1\le\ldots\le t_d\le 1;
$$
$$
t_1=\ldots=t_{b_1}<t_{b_1+1}=\ldots=t_{b_2}<\ldots<t_{b_{\ell-1}+1}=\ldots=t_{b_\ell}.
$$
Our improvement over \cite{FK} will be that the choice of breaking points will not only depend on $\{x_1,\ldots,x_d\}$, but also on $Y$. For $b\in \{1,\ldots,d\}$, define
\bee
\label{eq_cost_function}
\zeta_b:=2\sum_{i=1}^b\sum_{j=b+1}^d |Y_{ij}|^2
\ene
be the $\|\cdot\|_2^2$-cost of declaring $b$ a breaking point. In order to optimize the costs, the breaking points will be chosen in two steps. First, let
\bee
\label{eq_r_def}
\ell:=2\l\lceil\|[X,Y]\|_\nf^{1/3}\r\rceil^{-1},\quad r:=2(\ell-1)^{-1}.
\ene
and choose tentative breaking points $a_1,\ldots,a_{\ell-1}$ by splitting $[-1,1]$ into $\ell-1$ intervals of equal length and grouping $x_1,\ldots,x_d$ accordingly; in other words,
$$
x_1,\ldots,x_{a_1}\in -1+[0,r];
$$
$$
x_{a_1+1},\ldots,x_{a_2}\in -1+[r,2r];
$$
$$
\ldots
$$
$$
x_{b_{\ell-1}+1},\ldots,x_{b_\ell}=x_d\in -1+[(\ell-2)r,(\ell-1) r]=[1-r,1].
$$
The actual breaking points will be chosen by minimizing the cost function \eqref{eq_cost_function} over each of the intervals determined by tentative breaking points. That is, denoting for convenience $a_0:=0$,
\bee
\label{eq_def_aj}
b_j:=\arg\min\l\{\zeta_k\colon a_{j-1}+1\le k\le a_j\r\},\quad j=1,2,\ldots,\ell-1;\quad b_{\ell}:=d.
\ene
Clearly,
\bee
\label{eq_zeta_restatement}
\|Y-Y'\|_2^2=\sum_{j=1}^{\ell-1}\zeta_{b_j}.
\ene
In order to estimate the latter note, recall that, from Corollary \ref{cor_gaps}, we have assumed
$$
x_{b+1}-x_{b}< \sqrt{\frac{6}{C_{\mathrm{t1}}}}\|[X,Y]\|_\nf^{2/3}=:g,\quad \text{for}\quad b=1,2,\ldots,d-1;
$$
in other words, the length of every spectral gap of $X$ must be less than $g$. As a consequence, assuming, say, that $\ctone\ge 24$ so that $2g\le r$, we obtain
\bee
\label{eq_weights}
2\sum_{k=a_j+1}^{a_{j+1}}\frac{x_{k+1}-x_k}{r}=\frac{2(x_{a_{j+1}+1}-x_{a_j+1})}{r}\ge \frac{2(r-g)}{r}\ge 1,\quad \text{for}\quad j=0,\ldots,\ell-1.
\ene
Indeed, a violation of the second to last inequality would lead to a large gap between $x_{a_{j+1}}$ and $x_{a_{j+1}+1}$.

From \eqref{eq_def_aj}, we have that $\zeta_{b_j}$ cannot exceed any convex combination of $\zeta(k)$ from the right hand side of \eqref{eq_def_aj}, and the same applies to a linear combination with the coefficients from the left hand side of \eqref{eq_weights}. Thus, we have
$$
\zeta_{b_j}\le \frac{2}{r}\sum_{k=a_{j-1}+1}^{a_{j}}(x_{k+1}-x_k)\,\zeta_k,\quad j=1,\ldots,\ell-1.
$$
By combining with \eqref{eq_zeta_restatement} and recalling $a_\ell=d$, we then use \eqref{eq_cost_function} and expand
\begin{multline}
\label{eq_zeta_restatement_2}
\|Y-Y'\|_2^2\le \frac{2}{r}\sum_{k=1}^d (x_{k+1}-x_k)\zeta_k=\frac{4}{r}\sum_{k=1}^d (x_{k+1}-x_k)\sum_{i=1}^k\sum_{j=k+1}^d|Y_{ij}|^2=\\=\frac{4}{r}\sum_{i=1}^d\sum_{j=i+1}^d\l(\sum_{k=i}^{j-1}(x_{k+1}-x_k)\r)|Y_{ij}|^2=\frac{4}{r}\sum_{i=1}^d\sum_{j=i+1}^d(x_j-x_i)|Y_{ij}|^2.
\end{multline}
From the Cauchy -- Schwartz inequality applied to the right hand side, we now obtain
\bee
\label{eq_cauchy_schwartz}
\|Y-Y'\|_2^2\le \frac{4}{r}\l(\sum_{i=1}^d\sum_{j=i+1}^d(x_j-x_i)|Y_{ij}|^2\r)^{1/2}\l(\sum_{i=1}^d\sum_{j=i+1}^d|Y_{ij}|^2\r)^{1/2}\le \frac{2}{r}\|[X,Y]\|_2\|Y\|_2.
\ene
Since $\|Y\|_\nf\le \|Y\|\le 1$, we can multiply by $d^{-1}$ and recall the definition of $r$ \eqref{eq_r_def}, arriving to
\bee
\label{eq_conclusion_1}
\|Y-Y'\|_\nf^2\le \frac{2}{r}\|[X,Y]\|_\nf\le 4\|[X,Y]\|_\nf^{2/3}.
\ene
On the other hand, each diagonal entry $t_j$ of $X'$ is located in the same interval of length $r$ as the corresponding diagonal entry $x_j$ of $X$. Recalling the definition of $r$ \eqref{eq_r_def}, we have.
\bee
\label{eq_conclusion_2}
\|X-X'\|_\nf^2\le \|X-X'\|\le r\le 2\|[X,Y]\|^{1/3}_\nf.
\ene
The combination of \eqref{eq_conclusion_1} and \eqref{eq_conclusion_2} provides the conclusion of \ref{th_main} with $\ctone=8$. In view of the other restriction imposed earlier, Theorem \ref{th_main} holds with $\ctone=24$.\,\,\qed
\begin{rmk}
\label{rem_main_multiple}
The proof extends to the case of multiple matrices in Theorem \ref{th_main_multi1}, with the following modifications, which mostly replace every expression involving $Y$ by an appropriate summation over $Y_j$. In particular, \eqref{eq_cost_function} is replaced by
$$
\zeta_b:=\quad 2\sum_{q=1}^N\sum_{i=1}^b\sum_{j=b+1}^d |Y_{q,ij}|^2
$$
where $Y_{q,ij}$ is an element of the matrix $Y_q$. We also replace $\|[X,Y]\|_\nf$ in the same way as \eqref{eq_replacement} everywhere in the proof. In \eqref{eq_zeta_restatement_2}, we replace $Y$ and $Y'$ by $Y_q$ and $Y_q'$, respectively, and add summation over $q$ everywhere. The same is done in the left hand sides of \eqref{eq_cauchy_schwartz}, \eqref{eq_conclusion_1}, \eqref{eq_conclusion_2}, and in the summations in the middle of \eqref{eq_cauchy_schwartz}. The $2$-norms in the right hand side of \eqref{eq_cauchy_schwartz} are replaced similarly to \eqref{eq_replacement}. The upper estimate of $\|Y\|$ is replaced by that of $\l(\|Y_1\|_\nf^2+\ldots+\|Y_N\|_\nf^2\r)^{1/2}\le \sqrt{N}$, adding an extra factor $\sqrt{N}$ in the conclusion.
\end{rmk}
\section{Proof of theorem \ref{th_lower_bound}}
We first recall the original definition \eqref{eq_th_lower_bound}
$$
X_0 := \frac{1}{d} 
\begin{pmatrix}
1 &  &  &  \\
& 2 &  &  \\
&  & \ddots &  \\
&  &  & d
\end{pmatrix}
, \quad
Y_0 := 
\begin{pmatrix}
0 & 1 &  &  \\
1& 0 & \ddots &  \\
& \ddots & \ddots & 1 \\
&  & 1 & 0
\end{pmatrix}
$$
and note
$$
[X_0, Y_0] = \frac{1}{d} \begin{pmatrix}
0 & 1 &  &  \\
-1& 0 & \ddots &  \\
& \ddots & \ddots & 1 \\
&  & -1 & 0
\end{pmatrix}
$$
Clearly, $\frac{1}{d}\le \|[X_0,Y_0]\|_\nf\le \frac{2}{d}$. Let $X',Y'$ be commuting approximants satisfying the assumptions of Theorem \ref{th_lower_bound}, and $b_1,\ldots,b_\ell$ be the breaking points of $X'$. From Lemma \ref{lemma_reduction}, one can easily check (without carefully optimizing the constants)
$$
\|X_0-X'\|_2^2=\sum_{i=1}^{\ell}\sum_{j=b_{i-1}+1}^{b_i}\l(\frac{2j-(b_{i-1}+b_i+1)}{2d}\r)^2\ge \frac{1}{48 d^2}\sum_{i=1}^{\ell}(b_{i}-b_{i-1})^3\ge \frac{1}{48d^2}\frac{d^3}{\ell^2},
$$
where the last step involves the cubic mean inequality and the identity $b_\ell-b_0=d$.

On the other hand, $Y'$ is obtained from $Y$ by deleting $2(\ell-1)$ non-zero off-diagonal entries. Therefore,
$$
\|Y_0-Y'\|_2^2=2(\ell-1).
$$
By combining the above, we arrive to
$$
\|X_0-X'\|_2^2+\|Y_0-Y'\|_2^2\ge \frac{d}{48\ell^2}+\ell\ge \frac{1}{48}d^{1/3}\ge \frac{1}{96} d\|[X_d,Y_d]\|_\nf^{2/3}.
$$
After multiplying by $d^{-1}$, we arrive to the conclusion of Theorem \ref{th_lower_bound}.\qed

\end{document}